\documentclass[a4paper,11pt,fleqn]{article}

\usepackage[utf8]{inputenc}
\usepackage[T1]{fontenc}
\usepackage{textcomp}
\usepackage{amsmath,amssymb,amsthm,mathrsfs,txfonts,pxfonts}
\usepackage{lmodern}
\usepackage[a4paper]{geometry}
\usepackage{graphicx}
\usepackage[usenames,dvipsnames]{xcolor}
\usepackage{microtype}
\usepackage{enumerate}
\usepackage{hyperref}
\hypersetup{pdfstartview=XYZ}
\newtheoremstyle{theoreme}
 {\topsep} 
 {\topsep} 
 {} 
 {0pt} 
 {\bfseries} 
 {\newline} 
 {3pt} 
 {} 
\newtheoremstyle{proposition}
{\topsep} 
 {\topsep} 
 {} 
 {0pt} 
 {\bfseries} 
 {\newline} 
 {3pt} 
 {} 
 \newtheoremstyle{lemma}
 {\topsep} 
 {\topsep} 
 {} 
 {0pt} 
 {\bfseries} 
 {\newline} 
 {3pt} 
 {} 
\newtheoremstyle{definition}
 {\topsep}
 {\topsep}
 {}
 {0pt}
 {\bfseries}
 {\newline}
 {3pt}
 {}
\newtheoremstyle{remarque}
 {\topsep}
 {\topsep}
 {}
 {0pt}
 {\bfseries}
 {\newline}
 {3pt}
 {} 
\theoremstyle{theoreme}
\newtheorem{Thm}{Theorem}[section]
\newtheorem*{Thm*}{Theorem}

\theoremstyle{lemma}
\newtheorem{Lem}[Thm]{Lemma}
\theoremstyle{proposition}
\newtheorem{Prop}[Thm]{Proposition}
\newtheorem*{Prop*}{Proposition}

\theoremstyle{definition}

\theoremstyle{remarque}
\newtheorem{Rem}[Thm]{Remark}

\newenvironment{fonction}{\begin{array}{ccccc}}{\end{array}}

\newcommand\F{\mathbb{F}}

\newcommand\comment[1]{}

\begin{document}
\title{A Baer-Krull theorem for quasi-ordered groups}
\author{Salma Kuhlmann, Gabriel Lehéricy}
\maketitle

 \begin{abstract}
  We give group analogs of two important theorems of real algebra concerning convex valuations, one of which 
  is the Baer-Krull theorem. We do this by using quasi-orders, which gives a uniform approach to valued and ordered groups.
  We also recover the classical Baer-Krull theorem from its group analog.
 \end{abstract}

 \section*{Introduction}

  The theories of field ordering and field valuation present some strong similarities but are classically treated as 
  separate subjects. However, in \cite{Fakhruddin}, Fakhruddin found a way of unifying these two theories by using quasi-orders.
  He defined a quasi-ordered field as a field $K$ endowed with a total quasi-order (q.o) $\precsim$ satisfying the following axioms:\\
    $(Q_1)$ $\forall x(x\sim0\Rightarrow x=0)$\\
    $(Q_2)$ $\forall x,y,z(x\precsim y\nsim z\Rightarrow x+z\precsim y+z)$\\
    $(Q_3)$ $\forall x,y,z (x\precsim y\wedge 0\precsim z)\Rightarrow xz\precsim yz$\\
 He then showed that valued and ordered fields are particular instances of q.o fields, and even showed that they are the only ones,
 so that the theory of q.o fields is very convenient to unify the theories of valued fields with the theory of ordered fields.
 Two different notions of quasi-ordered groups have been studied in \cite{Lehericy} and \cite{Lehericy2}, both of which can be seen as a generalization of 
 ordered and valued groups. The group analog of Fakhruddin's theory was studied in \cite{Lehericy}.
 
  The notion of compatibility between a valuation and an order
  plays an important role in the theory of valued fields. One can give 
  several characterizations of compatibility of a valuation $v$ with respect to an order. The notion of 
  coarsening of a valuation also plays an important role in valuation theory and is in some sense 
  analogous to the notion of compatibility. In \cite{KuhlmannPoint}, the authors used Fakhruddin's theory to give 
  a uniform approach to the notions of compatible valuations with respect to an order and the notion of coarsening of 
  a valuation. We recall their result:
  \begin{Thm}\label{compforfields}
   Let $(K,\precsim)$ be a quasi-ordered field and $v$ a valuation on $K$. The following are equivalent:\\   
     1)$v$ is compatible with $\precsim$\\
     2)$R_v$ is $\precsim$-convex in $K$\\
     3)$I_v$ is $\precsim$-convex in $K$\\
     4)$I_v<1$\\
     5)$\precsim$ induces a  q.o on the residue field $R_v/I_v$.\\
  \end{Thm}
  In the special case where $\precsim$ is an order, Theorem \ref{compforfields} gives conditions for 
  a valuation $v$ to be compatible with $\leq$. If $\precsim$ comes from a valuation $w$ then Theorem 
  \ref{compforfields} gives conditions for $v$ to be a coarsening of $w$.
  
  An important result concerning 
  valuation-compatible orders is the Baer-Krull Theorem. Given a valued field 
  $(K,v)$, the Baer-Krull Theorem describes the set of $v$-compatible orderings of $K$  modulo 
   the set of orderings of the residue field. 
  We recall the Baer-Krull Theorem as it appears in \cite{EngelerPrestel}:
  
  \begin{Thm}[Baer-Krull]\label{baerkrull}
   Let $v:K\to G\cup\{\infty\}$ be a valuation and let $(\pi_i)_{i\in I}$ be a family of elements of $K$ such that 
   $(v(\pi_i)+2G)_{i\in I}$ is an $\mathbb{F}_2$-basis of $G/2G$. Let $\mathcal{O}(K)$ be the set of 
   of field orderings on $K$ with which $v$ is compatible and $\mathcal{O}(Kv)$ the set of orders of the residue field
   $Kv$.
   There is a bijection:   
   $\phi: \mathcal{O}(K)\longleftrightarrow\{-1,1\}^I\times\mathcal{O}(Kv)$ defined as follows: 
   $\phi(\leq)=(\eta,\leq_{Kv})$, where $\eta(i)=1$ if and only if $0\leq \pi_i$ and 
   $\leq_{Kv}$ is the order induced by $\leq$ on the residue field.   
   \end{Thm}

   The notion of valuation also exists for groups:
     a valuation on a
group G
 is a map $v$ from $G$ to a totally ordered set $(\Gamma\cup\{\infty\},\leq)$, where 
$\infty$ denotes the maximum of $\Gamma\cup\{\infty\}$, so that the following is satisfied:
$v(g)=\infty\Leftrightarrow g=0$, $v(g-h)\geq\min(v(g),v(h))$ and $v(g)\leq v(h)\Rightarrow v(x+g-x)\leq v(x+h-x)$
for all $g,h,x\in G$.
 Valued groups have been studied in \cite{PriessCrampe} and \cite{Kuhlmann}; however,
   the notions of compatible valuation on groups and of coarsening of group valuations still haven't been developed. 
   The goal
   of this paper is to introduce these notions using the theories of quasi-ordered groups developed in 
   \cite{Lehericy} and \cite{Lehericy2} and to
   find group analogs of Theorems \ref{compforfields} and 
   \ref{baerkrull}. In Section 1 we give some preliminaries on quasi-orders and prove the lemmas which we will need in the next
   sections.
    In Section 2 we state an analog of Theorem \ref{compforfields} in the framework of the compatible quasi-ordered groups defined 
    in \cite{Lehericy}, but conclude that there is no analog of the Baer-Krull Theorem in that case. In Section 3 we 
    consider the C-quasi-ordered groups defined in \cite{Lehericy2} and prove analogs of theorems \ref{compforfields} and 
    \ref{baerkrull} in that setting (Theorems \ref{theoremforCqo} and \ref{BaerKrullgroups}).  
     In Section 4 we explain how we can recover the classical Baer-Krull Theorem from 
    its group analog.

 \section{Quasi-orders on groups}\label{qogroups}
 
    We recall that a quasi-order is a binary relation which is reflexive and transitive.
For us, a q.o group is just a group endowed with a q.o with no 
    extra assumption. A q.o
    $\precsim$ naturally induces an equivalence
    relation which we denote by $\sim$ and defined as 
    $a\sim b\Leftrightarrow a\precsim b\precsim a$. We denote by $cl(a)$ the class of $a$ for this equivalence relation.
For any $a,b\in A$, the notation $a\precnsim b$ means $a\precsim b\wedge a\nsim b$.
    \textit{Every q.o and every order considered in this paper is total}.
     If $(A,\precsim)$ is a quasi-ordered set and $B\subseteq A$, we say that $B$ is
\textbf{$\precsim$-convex} (in $A$) if for any $a\in A$ and $b,c\in B$, $b\precsim a\precsim c$ implies 
$a\in B$. We say that $B$ is an \textbf{initial segment} (respectively, a \textbf{final segment}) of $A$ if for any 
$a\in A$ and $b\in B$, $a\precsim b$ (respectively, $b\precsim a$) implies $a\in B$.
 If $v:G\to\Gamma\cup\{\infty\}$ is a valuation on a group $G$, then for any $\gamma\in\Gamma$ we
 denote by $G^{\gamma}$ the set $\{g\in G\mid v(g)\geq\gamma\}$ and 
 by $G_{\gamma}$ the set $\{g\in G\mid v(g)>\gamma\}$.
Note that if $v$ is a valuation on $G$, then the relation 
    $g\precsim h\Leftrightarrow v(g)\geq v(h)$ is a total q.o, which we call the q.o induced by $v$.
We say that a q.o 
    is \textbf{valuational} if it is induced by a valuation.
    If $(G,\precsim)$ is a q.o group, we say that $g\in G$ is \textbf{v-type} if 
    $g\sim -g$ and \textbf{o-type} if $g\nsim -g\vee g=0$. The choice of this terminology will be made clear later
    (see Proposition \ref{vtypeotype}).
    We say that \textbf{the q.o $\precsim$ is o-type} (respectively \textbf{v-type}) if every element of $(G,\precsim)$ is 
    o-type (respectively v-type). In order to have a Baer-Krull
    theorem, we need to define notions of q.o induced on a quotient as well as notions of ``lifting'' q.o's from 
    quotients. Given a quasi-order $\precsim$ on $G$ and a subgroup $H$, we say that $\precsim$ induces a q.o on 
 $G/H$ if the following relation $\precsim$ defined on $G/H$ is transitive:
 $g_1+H\precsim g_2+H\Leftrightarrow\exists h_1,h_2\in H g_1+h_1\precsim g_2+h_2$. In that case we call this relation the
 \textbf{q.o induced by $\precsim$ on $G/H$}. Now assume we are given a valuation $v:G\to\Gamma$ and a family
 $(\precsim_{\gamma})_{\gamma\in\Gamma}$ of q.o's, each defined on a quotient $G^{\gamma}/G_{\gamma}$.
 We define \textbf{the lifting of $(\precsim_{\gamma})$ to $G$} as the q.o given by the formula:
 $g\precsim h\Leftrightarrow g+G_{\gamma}\precsim_{\gamma}h+G_{\gamma}$ where 
 $\gamma=\min(v(g),v(h))$. It is the only q.o such that for any $\gamma$, 
 $\precsim$ induces the q.o $\precsim_{\gamma}$ on $G^{\gamma}/G_{\gamma}$. Finally, if $(G,\precsim)$ is a q.o group and $v$ a 
 valuation on $G$, we say that $v$ is \textbf{compatible with $\precsim$}, or that $\precsim$ is a \textbf{$v$-compatible q.o}, 
 if for any $g,h\in G$, $0\precsim g\precsim h$ implies
 $v(g)\geq v(h)$.
 
 
 Towards an analog of Theorem \ref{compforfields},
 we will need the following lemmas. We fix a group $G$, an arbitrary total q.o $\precsim$ and a valuation 
 $v:G\to\Gamma\cup\{\infty\}$. 
 
 \begin{Lem}
   The following statements are equivalent:
  \begin{enumerate}[(1)]
   \item $\forall\gamma\in\Gamma,G^{\gamma}$ is $\precsim$-convex.
   \item $\forall\gamma\in\Gamma,G_{\gamma}$ is $\precsim$-convex.
  \end{enumerate}
 \end{Lem}

 \begin{proof}
  $(1)\Rightarrow(2)$:let $\gamma\in\Gamma$, $g,h\in G_{\gamma}$ and $f\in G$ with 
  $h\precsim f\precsim g$. Set $\delta:=\min(v(g),v(h))$. We have $\delta>\gamma$ so $G^{\delta}\subseteq G_{\gamma}$. 
  Moreover,
  $g,h\in G^{\delta}$, which by $\precsim$-convexity of $G^{\delta}$ implies $f\in G^{\delta}\subseteq G_{\gamma}$.
  $(2)\Rightarrow (1)$: Assume $g,h\in G^{\gamma}$. If $f\notin G^{\gamma}$ then 
   $v(f)<\gamma\leq v(g),v(h)$ hence $g,h\in G_{v(f)}$, which by assumption means we cannot have 
   $h\precsim f\precsim g$.
 \end{proof}

 \begin{Lem}
  $v$ is compatible with $\precsim$ if and only if $(G_{\gamma})^{\succsim 0}$ is $\precsim$-convex for all 
  $\gamma\in\Gamma$.
 \end{Lem}
\begin{proof}
 If $v$ is compatible with $\precsim$, then for any $0\precsim g\precsim h\in G_{\gamma}$ we have 
 $v(g)\geq v(h)>\gamma$ so $g\in G_{\gamma}$. Conversely, assume each $(G_{\gamma})^{\succsim 0}$ is $\precsim$-convex. If 
 $0\precsim g\precsim h$ and 
 $v(h)>v(g)$, then $h\in G_{v(g)}$ by definition and $g\in G_{v(g)}$ by $\precsim$-convexity of 
$(G_{\gamma})^{\succsim 0}$, which is a contradiction.
Thus, 
 $0\precsim g\precsim h\Rightarrow v(g)\geq v(h)$ and $v$ is compatible.
\end{proof}

\begin{Lem}\label{quotientlemma}
 Let $\precsim$ be a q.o such that:
 
 \[(\ast)\left\{\begin{array}{l}    cl(0)=\{0\}\\
                               g\precsim h\nsim -f\Rightarrow g+f\precsim h+f\text{ for any }g,h,f\in G.
                             \end{array}\right.\]
  and let $H$ be a subgroup of $G$. Then $H$ is $\precsim$-convex if and only if $\precsim$ induces a total q.o on $G/H$ such that
  $cl(0+H)=\{0+H\}$. Moreover, if $H$ is $\precsim$-convex, then for every $g\in G\backslash H$, $g\precsim -g$ if and only if 
  $g+H\precsim -g+H$; in particular, $g$
  is v-type (resp. o-type) if and only if 
  $g+H$ is v-type (resp. o-type).
\end{Lem}

\begin{proof}
 Assume $H$ is $\precsim$-convex. Let $f,g,h$ with $f+H\precsim g+H$ and $g+H\precsim h+H$. There are 
 $a,b,b',c\in H$ such that $f+a\precsim g+b$ and $g+b'\precsim h+c$.
 Assume first that $g,h\in H$. If $h\precsim f+a$ then by $\precsim$-convexity of 
 $H$ we have $f\in H$ and $f\precsim h+f-h$. Assume $h\notin H$, so $h+c\notin H$. By $\precsim$-convexity, 
 $h+c\nsim b-b'$, so by assumption $g+b'\precsim h+c$ implies 
 $f+a\precsim g+b=g+b'+b-b'\precsim h+c+b-b'$. If $g\notin H$ then the same reasoning gives 
 $f+a+b'-b\precsim g+b'\precsim h+c$. In any case, we have $f+H\precsim h+H$, so the relation $\precsim$ on $G/H$ is 
 transitive. This proves that $\precsim$ induces a q.o on $G/H$. If 
 $g+H\sim H$ then there are $a,b,c\in H$ with $a\precsim g+b\precsim c$ which implies $g\in H$, so $cl(H)=\{H\}$.
 Conversely, assume $\precsim$ induces a q.o on $G/H$ such that $cl(0+H)=\{0+H\}$.
 If $g,h\in H$ and $f\in G$ are such that
 $g\precsim f\precsim h$ then $0=g+H\precsim f+H\precsim h+H=0$ hence by assumption $f\in H$.
 
 For the second statement, let $g\in G\backslash H$. 
Assume $g\nprecsim -g$.  Since $\precsim$ is total, we have $-g\precsim g$. Note that 
$0\precsim g$: for if $0\nprecsim g$,  then since $\precsim$ is total we have 
$g\precsim 0\nsim g=-(-g)$ and $(\ast)$ implies 
$0\precsim -g$; so since $-g\precsim g$ we must have $0\precsim g$. Because $g\nprecsim -g$, we have in particular 
$g\nsim -g$, and so by $(\ast)$, \[g=0+g\precsim g+g=2g.\] Thus 
$0\precsim g\precsim 2g$. If $g+H\precsim -g+H$, then there are $a,b\in H$ with 
$g+a\precsim (-g)+b$. Since $H$ is $\precsim$-convex, and $a\in H$, and $-g+b\notin H$, we have 
$-(-a)=a\nsim -g+b$, so by $(\ast)$:
\[g\precsim -g+b-a.\] If $-g\sim -g+b-a$, then $g\precsim -g+b-a\precsim -g$, so we must have  $-g\nsim -g+b-a$, and by 
$(\ast)$ again 
\[2g\precsim b-a.\] Thus $0\precsim g\precsim 2g\precsim b-a$, and $g\in H$ because 
$H$ is $\precsim$-convex. This contradiction shows that $g+H\nprecsim -g+H$ if $g\nprecsim -g$.
\end{proof}
    
 \begin{Rem}
  The reader should pay attention to the fact that the second condition of $(\ast)$ is not equivalent to  axiom 
$(Q_2)$ of compatible q.o's given below: example 3.1(a) of \cite{Lehericy2} satisfies $(\ast)$ but not $(Q_2)$. However, it will later turn out that $(Q_1)$ and $(Q_2)$ 
imply $(\ast)$. The reason why we choose to state Lemma \ref{quotientlemma} with $(\ast)$ instead 
of $(Q_1)$ and $(Q_2)$ is the following: the class of q.o's satisfying $(\ast)$ contains both the class of compatible q.o's (see Section 2) and the class of C-q.o's 
(see Section 3), so that Lemma \ref{quotientlemma} can be applied to both cases.
 \end{Rem}
   
 \section{Compatible q.o's}
 
 A first idea to achieve the goal stated in the introduction 
 is to take inspiration from \cite{Fakhruddin}. We say that a q.o $\precsim$ is compatible 
 (with the group operation) if it satisfies:
 \begin{itemize}
    \item[$(Q_1)$] $cl(0)=\{0\}$
    \item[$(Q_2)$] $\forall x,y,z(x\precsim y\nsim z\Rightarrow x+z\precsim y+z)$ 
    \end{itemize}
    
 Group orders and valuational q.o's are particular cases of compatible q.o's. More precisely, we proved the
 following in \cite{Lehericy} (Proposition 2.8 and 2.15):
 \begin{Prop}\label{vtypeotype}
  Let $\precsim$ be a compatible q.o on an abelian group $G$. Then $\precsim$ is a group order if and only if every element is o-type, and
  $\precsim$ is valuational if and only if every element is v-type.
 \end{Prop}

 Compatible q.o's have been studied in \cite{Lehericy}. We proved there (Theorem 2.21) that if $\precsim$ is a compatible 
 q.o on an abelian group $G$, then there is a subgroup 
 $G^o$ of $G$ which is an initial segment of $G$ such that $\precsim$ is a group order on $G^o$ and $\precsim$ corresponds to a valuation outside of 
 $G^o$. It is easy to show that compatible q.o's satisfy the condition $(\ast)$ of lemma \ref{quotientlemma} and that 
 a subgroup $H$ of a compatible q.o group $(G,\precsim)$ is $\precsim$-convex if and only if $H^{\succsim 0}$ is $\precsim$-convex. Using the lemmas 
 from section \ref{qogroups},
 we get:

 \begin{Thm}
 Let $\precsim$ be a compatible q.o on an abelian group $G$ and $v:G\to\Gamma\cup\{\infty\}$ a valuation. 
 The following statements are equivalent:
  \begin{enumerate}
   \item $\forall\gamma\in\Gamma,G^{\gamma}$ is $\precsim$-convex.
   \item $\forall\gamma\in\Gamma,G_{\gamma}$ is $\precsim$-convex.
   \item $\forall\gamma\in\Gamma$, $\precsim$ induces a compatible q.o $\precsim_{\gamma}$
   on $G^{\gamma}/G_{\gamma}$.
   \item $v$ is compatible with $\precsim$.
  \end{enumerate}
  Moreover, $\precsim$ is valuational (respectively, an order) if and only if each $\precsim_{\gamma}$ is valuational
  (respectively, an order).
\end{Thm}
\begin{proof}
  All that remains to check is that the induced q.o on a quotient satisfies axiom $(Q_2)$; for this we refer to the end of the proof of 2.11 in  
  \cite{Lehericy}. The last part of the theorem is a consequence of Proposition \ref{vtypeotype} and Lemma 
  \ref{quotientlemma}.
\end{proof}

This deals with our first problem, i.e finding an analog of theorem \ref{compforfields} for groups. However, the class
of compatible q.o's is not appropriate for a Bear-Krull theorem.  This is due to the fact that the set of 
 o-type elements of a compatible q.o must be an initial segment of $G$. If we lift an arbitrary family 
 of compatible q.o's to $G$, then the set of o-type elements of $G$ won't be an initial segment in general, which means that 
 the class of compatible q.o's is not stable under lifting. 
 In order to state an analog of the Baer-Krull theorem, we need to
  consider another class of q.o's.

\section{C-q.o's and the Baer-Krull theorem}

We recall that a C-relation is a ternary relation satisfying the universal closure of the following axioms: 
$C(x,y,z)\Rightarrow C(x,z,y)\wedge\neg C(y,x,z)\wedge (C(w,y,z)\vee C(x,w,z))$ and $x\neq y\Rightarrow C(x,y,y)$.
A C-relation on a group $(G,+)$ is said to be \textbf{compatible} (with the group operation) if 
$C(f,g,h)\Rightarrow C(x+f+y,x+g+y,x+h+y)$ holds for every $f,g,h,x,y\in G$. See \cite{Delon} for more 
information on C-relations.
A C-q.o on a group $G$ is a q.o induced by a compatible C-relation, i.e $\precsim$ is a C-q.o if and only if there exists 
a compatible C-relation $C$ on $G$ such that 
$g\precsim h\Leftrightarrow\neg C(g,h,0)$ for every $g,h\in G$.  A structure theorem for groups endowed with a C-q.o was 
given in \cite{Lehericy2} (Theorem 3.41).
 The original purpose of C-q.o's is to study C-groups. This is possible thanks to a bijective correspondence between
 the class of compatible C-relations on a group and the class of C-q.o's. In particular, C-q.o's can be seen as a 
 generalization of both 
 orders and valuations on groups.
 
	    \begin{Prop}
                      Valuational q.o's are C-q.o's. Moreover, a C-q.o
                       $\precsim$ is valuational if and only if
                      every element of $(G,\precsim)$ is v-type. 
                      \end{Prop}

  In \cite{Fuchs}, the author showed that group orders are characterized by their positive cones: if we are given 
  a subset $P$ of $G$ such that $P\cap -P=\{0\}, P\cup -P=G$ and $P+P\subseteq P$,  then $P$ is
   the set of positive elements of the group order defined by $g\leq h\Leftrightarrow h-g\in P$.                 
 An order is not a C-q.o, but there is a natural connection between orders and o-type C-q.o's. More precisely, we have 
 the following:

 \begin{Prop}\label{mapomega}
  Assume $(G,\precsim)$ is an o-type C-q.o. Then $P:=\{g\in G\mid -g\precsim g\}$ is a positive cone and it is 
  stable under ``$\sim$''.
 \end{Prop}
 \begin{proof}
  We use the following result which was proved in \cite{Lehericy2} (Lemma 3.6.(i)): For any $g,h\in G$, $h\precnsim -g$ implies 
  $g\sim g+h$.
  Lemma 3.15 of \cite{Lehericy2} implies that $P$ is stable under ``$\sim$''.
  We clearly have $P\cap -P=\{0\}$ and $P\cup -P=G$. 
  We just have to show $P+P\subseteq P$. We actually show that $-P$ is closed under addition.
  Let $g,h\in -P$. If $h$ or $g$ is $0$ then obviously $g+h\in -P$, so we may assume that $g$ and $h$ are both
non-zero.  If $g\precsim h$ then $g\precnsim -h$ which implies
  $g+h\sim h$ so $g+h\in -P$. Otherwise we have $h\precnsim -g$ which means $g+h\sim g$ so $g+h\in -P$.
 \end{proof}

 Proposition \ref{mapomega} gives us a map 
 $\Omega:\{\text{o-type C-q.o's on $G$}\}\rightarrow\{\text{group orders on $G$}\}$, where $\Omega(\precsim)$ is the order 
 corresponding to the
 positive cone $\{g\in G\mid -g\precsim g\}$. This map is surjective: let $\leq$ be an order of $G$ with positive cone 
 $P$. Then we can define a q.o on $G$ as follows: declare that 
$0\precnsim g\precnsim h$ holds for every $g\in -P\backslash\{0\}$ and $h\in P\backslash\{0\}$; declare then that 
$g\sim h$ whenever $g,h\in -P\backslash\{0\}$ and that $g\precsim h\Leftrightarrow g\leq h$ whenever $g,h$ are both in $P\backslash\{0\}$.
 One can easily check that this is a C-q.o and that it is o-type. By construction it 
 obviously is a pre-image of $\leq$ under $\Omega$. 
   The fact that $\Omega$ is surjective allows us to see orders as special case of C-q.o's. Through $\Omega$ we will be able 
   to transform certain statements concerning C-q.o's into statements about orders. In other words, we can use C-q.o's as a uniform 
   approach to ordered and valued groups. 

   We proved in \cite{Lehericy2} (Proposition 3.9) that C-q.o's satisfy the condition $(\ast)$ of Lemma \ref{quotientlemma}. 
   Moreover, we know that if 
   $\precsim$ is a C-q.o then a subgroup $H$ is $\precsim$-convex if and only if $H^{\succsim0}$ is 
   $\precsim$-convex. This allows us to apply the lemmas from section \ref{qogroups} to the case of C-q.o's:
\begin{Thm}\label{theoremforCqo}
 Let $\precsim$ be a C-q.o on $G$ and $v:G\to\Gamma\cup\{\infty\}$ a valuation. 
 The following statements are equivalent:
  \begin{enumerate}
   \item $\forall\gamma\in\Gamma,G^{\gamma}$ is $\precsim$-convex.
   \item $\forall\gamma\in\Gamma,G_{\gamma}$ is $\precsim$-convex.
   \item $\forall\gamma\in\Gamma$, $\precsim$ induces a C-q.o $\precsim_{\gamma}$.
   on $G^{\gamma}/G_{\gamma}$.
   \item $v$ is compatible with $\precsim$.
  \end{enumerate}
  Moreover, $\precsim$ is v-type (respectively o-type) if and only if each $\precsim_{\gamma}$ is v-type 
  (respectively o-type).
\end{Thm}
\begin{proof}
 We just have to show that the induced q.o on the quotient $G^{\gamma}/G_{\gamma}$ is a
 C-q.o. For this, we refer to the end of the proof of Proposition 3.11 in \cite{Lehericy2}.
\end{proof}

Unlike compatible q.o's, we proved in \cite{Lehericy2} (Proposition 3.38) that the class of C-q.o's is stable under lifting. This allows us to state 
a Baer-Krull theorem for C-q.o's:

\begin{Thm}[Baer-Krull for C-q.o's]\label{BaerKrullgroups}
 Let $(G,v)$ be a valued group. The map:
 
 $\begin{fonction}
&\{\text{C-q.o}\precsim\mid v\text{ is compatible with }\precsim\}&\longleftrightarrow&
 \{\text{ family of C-q.o's}(\precsim_{\gamma})_{\gamma\in\Gamma}\}\\
&\precsim &\longmapsto & (\precsim_{\gamma})_{\gamma\in\Gamma},\end{fonction}$
 
 where $\precsim_{\gamma}$ denotes the q.o induced by $\precsim$ on the quotient $G^{\gamma}/G_{\gamma}$, is a bijection.
 Moreover, $\precsim$ is v-type (respectively, o-type) if and only if each $\precsim_{\gamma}$ is 
 v-type (respectively o-type).
\end{Thm}

\begin{proof}
 We know from Proposition 3.38 of \cite{Lehericy2} that we can lift a family $(\precsim_{\gamma})_{\gamma\in\Gamma}$ of C-q.o's to obtain a C-q.o 
 $\precsim$ on $G$, and $\precsim$ is obviously compatible with $v$ by definition of the lifting. Conversely,
 if $\precsim$ is a C-q.o such that $v$ is compatible with $\precsim$, then by Theorem \ref{theoremforCqo}, we know 
 that $\precsim$ induces a C-q.o $\precsim_{\gamma}$ on each $G^{\gamma}/G_{\gamma}$.
 It is easy to see that these operations are inverse to each other.
\end{proof}

Since valuational q.o's are in particular C-q.o's, we have as an immediate corollary:
\begin{Thm}
 Let $(G,v)$ be a valued group. The map:
 
 $\begin{fonction}&\{\text{valuations }w\mid v\text{ is a coarsening of }w\}&\longleftrightarrow&
 \{\text{ family of valuations }(w_{\gamma})_{\gamma\in\Gamma}\}\\
 & w&\longmapsto&(w_{\gamma})_{\gamma\in\Gamma}, \end{fonction}$

 where  $w_{\gamma}:G^{\gamma}/G_{\gamma}\to v(G^{\gamma}\backslash G_{\gamma})\cup\{\infty\}$ is  defined  by 
$w_{\gamma}(g+G_{\gamma}):=v(g)$ for any $g\in G^{\gamma}\backslash G_{\gamma}$ and $w_{\gamma}(0):=\infty$, is a bijection.
\end{Thm}

Thanks to Lemma \ref{quotientlemma}, we see that if $\precsim$ is o-type and 
    and if $H$ is a $\precsim$-convex subgroup of $G$, then 
   the image under $\Omega$ of the q.o induced on $G/H$ by $\precsim$ is exactly the order induced by 
   $\Omega(\precsim)$ on $G/H$. This leads
   to another corollary of Theorem \ref{BaerKrullgroups}:

\begin{Thm}\label{Thmorders}
 Let $(G,v)$ be a valued group. The map:
 
 $\begin{fonction}
&\{\text{orders }\leq\mid v\text{ is compatible with }\leq\}&\longleftrightarrow&
 \{\text{ family of orders }(\leq_{\gamma})_{\gamma\in\Gamma}\}\\
& \leq&\longmapsto& (\leq_{\gamma})_{\gamma\in\Gamma},\end{fonction}$

where each $\leq_{\gamma}$ is the order induced by $\leq$ on the quotient $G^{\gamma}/G_{\gamma}$, is a bijection.
\end{Thm}
 \begin{proof}
   For each $\gamma$, let 
  $\precsim_{\gamma}\in\Omega^{-1}(\leq_{\gamma})$. By Theorem \ref{BaerKrullgroups}, we can lift 
  $(\precsim_{\gamma})_{\gamma}$ to $G$ and obtain a C-q.o $\precsim$. Since each $\precsim_{\gamma}$ is o-type, so is $\precsim$, 
  so we can define $\leq$ as $\Omega(\precsim)$. Then the order induced by $\leq$ on each $G^{\gamma}/G_{\gamma}$ is $\leq_{\gamma}$.
 \end{proof}

 \section{q-sections and the classical Baer-Krull Theorem}
 Finally, we want to show how one can recover the classical Baer-Krull theorem from Theorem \ref{Thmorders}. 
 We fix a valued field $(K,v)$ with value group $(G,\leq)$.
  Note that if $(\leq_g)_{g\in G}$ is a family of group orders 
 on the quotients $K^g/K_g$, then Theorem \ref{Thmorders} only tells us that this family lifts to a group order on $(K,+)$, but 
 there is no reason to think that this lifting is a field order in general. In order to understand the connection between 
 Theorem \ref{Thmorders} and the classical Baer-Krull theorem, we need to characterize the families $(\leq_g)_{g\in G}$ whose
 lifting to $K$ is a field order. 
 
  We can achieve this by using the the notion of q-section  developed in \cite{Prestel}. A q-section of the valued 
 field $(K,v)$ is a map $s:G\to K$ such that $s(0)=1$, $v(s(g))=g$ and 
 $s(g+h)\equiv s(g)s(h)$ mod $K^2$. It was proved in \cite{Prestel} that every valued field admits a q-section. 
 We now fix
 a q-section $s$ of $(K,v)$. Then for any $g\in G$, the map
  $\phi_g: Kv\to K^g/K_g,a+K_0\mapsto s(g)a+K_g$ defines an isomorphism from $Kv$ to $K^g/K_g$. If we take a
  family $(\leq_g)_{g\in G}$ of orders on the quotients $K^g/K_g$ then the behavior of the $\phi_g$'s with respect to 
  $\leq_g$'s will tell us if the lifting of $(\leq_g)_g$ is a field order:
 
   \begin{Prop}\label{fieldorders}
    Let $(\leq_g)_{g\in G}$ be a family of group orders on the quotients $K^g/K_g$. 
     Then the lifting of $(\leq_g)_g$ to $K$ is a field 
    order if and only if the following conditions are satisfied:
    
      \begin{enumerate}
       \item $\leq_0$ is a field order of $Kv$
       \item there exists a  group homomorphism 
    $\epsilon: G\to\{-1,1\}$ such that for any $g\in G$, $\phi_g$ is order-preserving when $\epsilon(g)=1$ and 
    $\phi_g$ is order-reversing when $\epsilon(g)=-1$.
      \end{enumerate}

   \end{Prop}

   \begin{proof}
  
    Denote by $\leq$ the lifting. If $\leq$ is a field order, then $\leq_0$ must be a field order because it is the order
    induced by $\leq$ on $Kv$; moreover, we can define 
    $\epsilon(g)=1$ if $0<s(g)$ and $\epsilon(g)=-1$ if $s(g)<0$, and one easily sees that
    $\epsilon$ has the desired property.
    Assume now that conditions (1) and (2) are satisfied.
    We already know from Theorem \ref{Thmorders} that $\leq$ is a group order of $(K,+)$, so $\leq$ 
    is a field order if and only if
     the set of positive 
    elements of $(K,\leq)$ is stable under multiplication. By definition of $\leq$, this is equivalent to saying that 
    for any $a,b\in K$ with $g=v(a)$ and $h=v(b)$, $0\leq_ga+K_g$ and $0\leq_hb+K_h$ imply 
    $0\leq_{g+h}ab+K_{g+h}$. Note that since $s$ is a q-section, we have 
    $ab+K_{g+h}=d^2\phi_{g+h}(\phi_g^{-1}(a+K_g)\phi_h^{-1}(b+K_h))$ for some $d\in K$; in particular, 
    $ab+K_{g+h}$ has the same sign as $\phi_{g+h}(\phi_g^{-1}(a+K_g)\phi_h^{-1}(b+K_h))$. Now assume for example that $\phi_g$ is 
    order-preserving and $\phi_h$ order-reversing.  If $a+K_g$ and $b+K_h$ are both positive,we then have 
    $\phi_g^{-1}(a+K_g)\leq_00$ and $0\leq_0\phi_h^{-1}(b+K_h)$. Since 
    $\leq_0$ is a field order, this implies that $\phi_g^{-1}(a+K_g)\phi_h^{-1}(b+K_h)\leq_00$.    
    Since $\epsilon$ is a group homomorphism, then 
    $\phi_{g+h}$ is order-reversing, hence 
    $0\leq_{g+h}\phi_{g+h}(\phi_g^{-1}(a+K_g)\phi_h^{-1}(b+K_h))$. The other cases are treated similarly.
        
   \end{proof}

  As a consequence of proposition \ref{fieldorders} we have the following variant of the Baer-Krull theorem:
   \begin{Thm}[Baer-Krull, variant]\label{Baerkrullvariant}
    Let  $\mathcal{O}$ be the set of field orders of $Kv$ and $\mathcal{E}$ the set of group homomorphisms from 
    $G$ to $\{-1,1\}$. Then $\mathcal{O}\times\mathcal{E}$ is in bijection with the set of families 
    $(\leq_g)_{g\in G}$ whose lifting to $K$ is a field order.
   \end{Thm}
    \begin{proof}
     Assume that $\leq_0\in\mathcal{O}$ and $\epsilon\in\mathcal{E}$ are given.
     Then define $\leq_g$ on $K^g/K_g$ as follows:
      If $\epsilon(g)=1$, define $a+K_g\leq_g b+K_g\Leftrightarrow \phi_g^{-1}(a+K_g)\leq_0\phi_g^{-1}(b+K_g)$;
      if $\epsilon(g)=-1$, define $a+K_g\leq_g b+K_g\Leftrightarrow \phi_g^{-1}(a+K_g)\geq_0\phi_g^{-1}(b+K_g)$.                            
      By the previous proposition, this family of orders lifts to a field order on $K$.
    \end{proof}

    Now assume that $(\pi_i)_{i\in I}$ is a family of elements of $K$ such that 
    $(v(\pi_i)+2G)_{i\in I}$ is an $\F_2$-Basis of $G/2G$. In order to recover Theorem \ref{baerkrull} from Theorem
    \ref{Baerkrullvariant}, we need to show that the set of homomorphisms from $G$ to $\{-1,1\}$ is in bijection 
    with the set of maps from $I$ to $\{-1,1\}$. First note that we can see $I$ as a subset of $G$ if we identify 
    $i\in I$ with $v(\pi_i)$. Thus, any homomorphism $\epsilon:G\to\{-1,1\}$ canonically induces a map 
    $I\to \{-1,1\}$ (just take $\epsilon_{\mid I}$). For the converse, note that every $g\in G$ has a decomposition 
    $g=\sum_{i\in I}n_iv(\pi_i)+2h$, where $h\in G$, $n_i\in\{0,1\}$ and $n_i=1$ only for finitely many $i$. If 
    $\epsilon:I\to\{-1,1\}$ is given, we can extend $\epsilon$ to a homomorphism $G\to\{-1,1\}$ as follows: 
    declare that $\epsilon(g)=1$ if the number of $i\in I$ such that $n_i=1$ in the decomposition of $g$ is even, and declare
    $\epsilon(g)=-1$ if it is odd.

{\small FACHBEREICH MATHEMATIK UND STATISTIK, 

UNIVERSITÄT KONSTANZ,

78457, GERMANY.

\emph{Email address:} salma.kuhlmann@uni-konstanz.de}\\

{\small FACHBEREICH MATHEMATIK UND STATISTIK, 

UNIVERSITÄT KONSTANZ,

78457, GERMANY.

\emph{Email address:} gabriel.lehericy@uni-konstanz.de}

\end{document}